\documentclass[a4paper,12pt,leqno]{article}
\usepackage[margin=3cm]{geometry}
\usepackage{amsmath}
\usepackage{amssymb}
\usepackage{amsxtra}
\usepackage{amsthm}
\usepackage{bm}
\usepackage[pagebackref,breaklinks]{hyperref}
\usepackage[T1]{fontenc}
\usepackage{lmodern}
\usepackage[lite,abbrev,shortalphabetic]{amsrefs}
\usepackage{colonequals}
\usepackage[shortlabels]{enumitem}
\setlist[enumerate,1]{(i)}

\newcommand{\Kl}{\mathrm{Kl}}

\newcommand{\C}{\mathbf{C}}
\newcommand{\F}{\mathbf{F}}
\newcommand{\bone}{\mathbf{1}}
\newcommand{\cA}{\mathcal{A}}
\newcommand{\cB}{\mathcal{B}}
\newcommand{\cC}{\mathcal{C}}
\newcommand{\cX}{\mathcal{X}}

\newcommand{\etab}{\bm{\eta}}
\newcommand{\rhob}{\bm{\rho}}

\theoremstyle{definition}

\newtheorem{remark}[section]{Remark}
\theoremstyle{plain}

\newtheorem{theorem}[section]{Theorem}
\newtheorem{cor}[section]{Corollary}

\begin{document}
\title{On the distribution of multivariate Jacobi sums}
\author{Qing Lu\thanks{School of Mathematical Sciences, Beijing Normal University, Beijing
100875, China; email: \texttt{qlu@bnu.edu.cn}. Partially supported by
Beijing Natural Science Foundation Grant 1202014; Fundamental Research Funds
for Central Universities of China; China Scholarship Council.}\and Weizhe
Zheng\thanks{Morningside Center of Mathematics and Hua Loo-Keng Key
Laboratory of Mathematics, Academy of Mathematics and Systems Science,
Chinese Academy of Sciences, Beijing 100190, China; University of the
Chinese Academy of Sciences, Beijing 100049, China; email:
\texttt{wzheng@math.ac.cn}. Partially supported by National Natural Science
Foundation of China Grants 11822110, 11688101, 11621061.}}
\date{}
\maketitle

\begin{abstract}
Let $\F_q$ be a finite field of $q$ elements. We show that the normalized
Jacobi sum $q^{-(m-1)/2}J(\chi_1,\dots,\chi_m)$ ($\chi_1\dotsm \chi_m$
nontrivial) is asymptotically equidistributed on the unit circle, when
$\chi_1\in \cA_1,\dots, \chi_m\in \cA_m$ run through arbitrary sets of
nontrivial multiplicative characters of $\F_q^\times$, if $\#\cA_1\ge
q^{\frac{1}{2}+\epsilon}$, $\#\cA_2 \ge (\log q)^{\frac{1}{\delta}-1}$ for
$\epsilon>\delta>0$ fixed and $q\to \infty$ or if $\#\cA_1\#\cA_2/q\to
\infty$. This extends previous results of Xi, Z.~Zheng, and the authors.
\end{abstract}

Let $\F_q$ be a finite field of $q$ elements and let $\C$ be the field of
complex numbers. We let $\cX_q$ denote the set of nontrivial multiplicative
characters $\F_q^\times\to \C^\times$. For $m\ge 2$, $\chi_1,\dots,\chi_m\in
\cX_q$ satisfying $\chi_1\dotsm \chi_m\neq \bone$, we consider the Jacobi
sum
\[J(\chi_1,\dots,\chi_m)=\sum_{\substack{a_1,\dots,a_m\in \F_q^\times\\
a_1+\dots+a_m=1}}\chi_1(a_1)\dotsm\chi_m(a_m),
\]
which has absolute value $q^{\frac{m-1}{2}}$. For a nonempty subset
$\cA\subseteq \cX^m_q$, we define the \emph{discrepancy} $D(\cA)$, to be the
supremum of
\[\left\lvert
\frac{\#\{(\chi_1,\dots,\chi_m)\in \cA^\circ\mid
q^{-\frac{m-1}{2}}J(\chi_1,\dots,\chi_m)\in e^{2\pi
i[a,b]}\}}{\#\cA^\circ}-(b-a)\right\rvert
\]
for all real numbers $a\le b\le a+1$, where
\[\cA^\circ=\{(\chi_1,\dots,\chi_m)\in
\cA\mid \chi_1\dotsm \chi_m\neq \bone\}
\]
and $e^{2\pi i[a,b]}$ denotes the image of the interval $[a,b]$ under the
map $x\mapsto e^{2\pi i x}$. For $\cA^\circ$ empty we adopt the convention
that $D(\cA)=1$.

The goal of this note is to prove the following equidistribution property.

\begin{theorem}\label{t.D}
There exists a constant $C$ such that for all $q$ and all integers $m\ge 2$,
$s\ge 1$ and for all nonempty subsets $\cA_1\subseteq \cX_q$,
$\cA_2\subseteq \cX_q$, $\cB\subseteq\cX_q^{m-2}$, we have
\begin{gather}
\label{e.0}D(\cA_1\times \cA_2\times \cB)\le C\left(s\left(\frac{\sqrt q}{A_1}\right)^{\frac{1}{2s+1}}+\frac{1}{\sqrt{sA_2}}\left(\frac{q}{A_1}\right)^{\frac{1}{2s}}\log q\right),\\
\label{e.1}D(\cA_1\times \cA_2\times \cB)\le C\left(\frac{q}{A_1A_2}\right)^{\frac{1}{4}}.
\end{gather}
Here $A_i=\#\cA_i$.
\end{theorem}

\begin{cor}\label{c}
\begin{enumerate}
\item Let $0<\epsilon \le\frac{1}{2}$ and $c>0$ be real constants. Then
    for nonempty subsets $\cA_1\subseteq \cX_q$, $\cA_2\subseteq \cX_q$,
    $\cB\subseteq\cX_q^{m-2}$ satisfying $\# \cA_1\ge
cq^{\frac{1}{2}+\epsilon}$, we have $D(\cA_1\times\cA_2\times \cB)\to 0$
as
\[\frac{\#\cA_2}{(\log
q)^{\frac{1}{\epsilon}-1}(\log\log q)}\to
 \infty.
 \]
\item For nonempty subsets $\cA_1\subseteq \cX_q$, $\cA_2\subseteq \cX_q$,
    $\cB\subseteq\cX_q^{m-2}$, $D(\cA_1\times\cA_2\times \cB)\to 0$ as
    $\#\cA_1\#\cA_2/q\to \infty$.
\end{enumerate}
\end{cor}

\begin{remark}
\begin{enumerate}
\item For $m=2$, the problem of asymptotic equidistribution of
    \[\{q^{-1/2}J(\chi_1,\chi_2)\}_{\chi_1\in \cA_1,\ \chi_2\in \cA_2,\
    \chi_1\chi_2\neq \bone}\] was first suggested by Shparlinski
    \cite[Section~5]{Shp}.

\item The case $s=1$ of \eqref{e.0} was proved in \cite[(1.2)]{LZZ}, which
    implies $D(\cA_1\times\cA_2\times \cB)\to 0$ as $\frac{A_1
    A_2}{q\log^2 q}\to \infty$. The case $m=2$ of \eqref{e.0} is a theorem
    of Xi \cite[Theorem 1.5]{Xi} and the problem of generalization to
    multivariate Jacobi sums was suggested by him in \cite{Xi}.

\item The bound \eqref{e.1} improves \cite[(1.3)]{LZZ}. Note that
    \eqref{e.1} is stronger than \eqref{e.0} when $(A_1 q)^{1/3}\le A_2\le
    A_1$.

\item It follows from a theorem of Katz \cite[Theorem 9.5]{Katz} (see also
    \cite[(1.7)]{LZZ} for the case $m=2$) that for fixed
    $\chi_2,\dots,\chi_m\in \cX_q$ satisfying $\chi_2\dotsm \chi_m\neq
    \bone$,
\[\{q^{-(m-1)/2}J(\chi_1,\dots,\chi_m)\}_{\chi_1\in
\cX_q,\ \chi_1\dotsm \chi_m\neq \bone}
\]
is asymptotically equidistributed on the unit circle as $q\to \infty$.

\item When one or both of the $\cA_i$'s are $\cX_q$, better bounds than
    Theorem~\ref{t.D} have been given. See \cite{KZ} by Katz and Z.~Zheng
    (for the case $m=2$ and $\cA_1=\cA_2=\cX_q$) and \cite[Theorems 1.4,
    1.6]{LZZ}.
\end{enumerate}
\end{remark}

Theorem \ref{t.D} follows from the following bounds of the moments of Jacobi
sums, which we deduce from results of Deligne and Katz. Our proof of
\eqref{e.M1} is similar to that of Xi \cite[Theorem 1.4]{Xi}, using
H\"older's inequality. Our proof of \eqref{e.M2} using Cauchy's inequality
was originally presented using matrices indexed by $\widehat{\F_q^\times}$.
Xi suggested the current presentation, similar to the proof of \cite[Lemma
3.8]{dSSV}.

\begin{theorem}\label{t.M}
For all integers $m\ge 2$, $s\ge 1$, $n\ge 1$ and for all nonempty subsets
$\cA_1\subseteq \cX_q$, $\cA_2\subseteq \cX_q$, $\cB\subseteq\cX_q^{m-2}$,
we have
\begin{gather}
\label{e.M1}\lvert M^{(n)}\rvert\le A_1^{1-\frac{1}{2s}}(s!A_2^sq+sA_2^{2s}(n\sqrt q+1))^{\frac{1}{2s}}B,\\
\label{e.M2}\lvert M^{(n)}\rvert \le n\sqrt{A_1A_2q}B,
\end{gather}
where $M^{(n)}=\sum_{(\chi_1,\dots,\chi_m)\in \cA^\circ}
(q^{-\frac{m-1}{2}}J(\chi_1,\dots,\chi_m))^n$,
$\cA=\cA_1\times\cA_2\times\cB$, $A_i=\#\cA_i$, $B=\#\cB$.
\end{theorem}

\begin{proof}
We have
\[M^{(n)}=\sum_{(\chi_3,\dots,\chi_m)\in \cB}\sum_{\substack{\chi_1\in \cA_1,\ \chi_2\in \cA_2\\(\chi_1,\dots,\chi_m)\in \cA^\circ}}
(q^{-\frac{m-1}{2}}J(\chi_1,\dots,\chi_m))^n,
\]
It suffices to prove the stated estimates for the inner sum. Thus we may
assume $B=1$. Let $\cB=\{(\chi_3,\dots,\chi_m)\}$ and $\lambda=\chi_3\dotsm
\chi_m$. Fix a nontrivial additive character $\psi\colon \F_q \to
\C^\times$. For $\chi\in \widehat{\F_q^\times}=\cX_q\cup \{\bone\}$,
consider the Gauss sum
\[G(\chi)=\sum_{a\in \F_q^\times} \psi(a)\chi(a).\]
Recall that $\lvert G(\chi)\rvert=\sqrt q$ for $\chi\in \cX_q$ and
$G(\bone)=-1$. We have
\[J(\chi_1,\dots,\chi_m)=\frac{G(\chi_1)\dotsm G(\chi_m)}{G(\chi_1\dotsm \chi_m)}.\]

By the above, we have
\[\lvert M^{(n)}\rvert\le \sum_{\chi_1\in \cA_1}\left\lvert\sum_{\substack{\chi_2\in \cA_2\\\chi_1\chi_2\lambda\neq \bone}} q^{-n}G(\chi_2)^n\overline{G(\chi_1\chi_2\lambda)}^n\right\rvert.\]
By H\"older's inequality, we get
\begin{align*}
\lvert M^{(n)}\rvert^{2s} &\le A_1^{2s-1}\sum_{\chi_1\in \widehat{\F_q^\times}}\left\lvert\sum_{\substack{\chi_2\in \cA_2\\\chi_1\chi_2\lambda\neq \bone}} q^{-n}G(\chi_2)^n\overline{G(\chi_1\chi_2\lambda)}^n\right\rvert^{2s}\\
&=A_1^{2s-1}\sum_{\etab,\rhob\in \cA_2^s} S(\etab,\rhob)\prod_{i=1}^s q^{-n}\overline{G(\eta_i)}^n G(\rho_i)^n,
\end{align*}
where
\[S(\etab,\rhob)=\sum_{\chi\in \widehat{\F_q^\times}\backslash \cC_{\etab,\rhob} }\prod_{i=1}^s q^{-n}G(\chi\eta_i\lambda)^n\overline{G(\chi\rho_i\lambda)}^n,\]
Here $\cC_{\etab,\rhob}=\{\eta_1^{-1}\lambda^{-1},\dots,
\eta_s^{-1}\lambda^{-1},\rho_1^{-1}\lambda^{-1},\dots,\rho_s^{-1}\lambda^{-1}\}$.
Each summand of $S(\etab,\rhob)$ has absolute value $1$, which implies the
trivial bound $\lvert S(\etab,\rhob)\rvert\le q-2$. If $\rhob$ is not a
permutation of $\etab$, then, removing the redundant factors in the product
and applying Katz's bound \cite[page 162, line~8]{Katz} to the sum extended
to $\widehat{\F_q^\times}$, we get
\[\lvert S(\etab,\rhob)\rvert\le \left\lvert \sum_{\chi\in \widehat{\F_q^\times}}\right\rvert+\left\lvert \sum_{\chi\in \cC_{\etab,\rhob}}\right\rvert\le sn\frac{q-1}{\sqrt{q}}+\frac{2s}{\sqrt{q^n}},\]
where we used the fact $\# \cC_{\etab,\rhob}\le 2s$. Therefore,
\[\lvert M^{(n)}\rvert^{2s}\le A_1^{2s-1}(s! A_2^s q+A_2^{2s} s(n\sqrt{q}+1)).\]

For the proof of \eqref{e.M2}, let
\[
\tilde M^{(n)}=\sum_{\substack{\chi_1\in \cA_1\\\chi_2\in \cA_2}}q^{-\frac{(m+1)n}{2}}G(\chi_1)^n\dotsm G(\chi_m)^n\overline{G(\chi_1\chi_2\lambda)}^n.
\]
Then
\[\lvert \tilde M^{(n)}-M^{(n)}\rvert \le \min\{A_1,A_2\}q^{-n/2}\le \sqrt{A_1A_2}q^{-n/2}\]
and
\[\tilde M^{(n)}=q^{-\frac{(m+1)n}{2}}G(\chi_3)^n\dots G(\chi_m)^n\sum_{\substack{\chi_1\in
\cA_1\\\chi_2\in \cA_2}}\sum_{a\in
\F_q^\times}G(\chi_1)^nG(\chi_2)^n\overline{\chi_1}(a)\overline{\chi_2}(a)\overline{\lambda}(a)\overline{\Kl_n(a)}.
\]
Here we used
\[G(\chi)^n=\sum_{a\in \F_q^\times}\Kl_n(a)\chi(a),\]
where
\[\Kl_n(a)=\sum_{\substack{a_1,\dots,a_n\in \F_q^\times\\a_1\dotsm
a_n=a}} \psi(a_1+\dots+a_n)
\]
is the Kloosterman sum. Applying Deligne's bound \cite[(7.1.3)]{Deligne}
\[\lvert \Kl_n(a)\rvert \le nq^{(n-1)/2},\]
we get
\[\lvert\tilde M^{(n)}\rvert \le nq^{-1/2}\sum_{a\in \F_q^\times}\lvert g_1(a)g_2(a)\rvert,\]
where $g_i(a)=\sum_{\chi\in \cA_i}q^{-n/2}G(\chi)^n\overline{\chi}(a)$.
Since
\[\sum_{a\in \F_q^\times}\lvert g_i(a)\rvert^2=\sum_{a\in \F_q^\times}\sum_{\eta,\rho\in \cA_i}q^{-n}\overline{G(\eta)}^nG(\rho)^n\eta\overline{\rho}(a)=A_i(q-1),\]
we have
\[\sum_{a\in \F_q^\times}\lvert g_1(a)g_2(a)\rvert\le \sqrt{A_1A_2}(q-1)\]
by Cauchy's inequality. Therefore,
\[\lvert M^{(n)}\rvert\le \sqrt{A_1A_2}q^{-1/2}(n(q-1)+1).\]
\end{proof}

\begin{proof}[Proof of Theorem \ref{t.D}]
Given nonnegative functions $f$ and $g$, we write $f\ll g$ for the existence
of an \emph{absolute} constant $C_0$ such that $f\le C_0 g$. Let
$D=D(\cA_1\times\cA_2\times \cB)$. We apply the Erd\H os-Tur\'an inequality
\cite[Theorem III]{ET}. For any real number $K\ge 1$, we have
\[D\ll \frac{1}{K}+\sum_{1\le n\le K} \frac{M^{(n)}}{n\#\cA^\circ}.\]
Note that $\#\cA^\circ\ge (A_1-1)A_2B$.

For \eqref{e.0}, taking $C\ge 2$, we may assume $A_1\ge (2s)^{2s+1} \sqrt q$
since $D\le 1$. By \eqref{e.M1}, we get
\begin{align}
\notag D&\ll \frac{1}{K}+\sum_{1\le n\le K}\left(\frac{\sqrt s}{n} \left(\frac{q}{A_1A_2^s}\right)^{\frac{1}{2s}}+\frac{1}{n}\left(\frac{sn\sqrt q}{A_1}\right)^{\frac{1}{2s}}\right)\\
\label{e.D}&\ll \frac{1}{K}+\sqrt s\left(\frac{q}{A_1A_2^s}\right)^{\frac{1}{2s}}(1+\log K) +s\left(\frac{sK\sqrt q}{A_1}\right)^{\frac{1}{2s}}.
\end{align}
Taking $K=s^{-1}\left(\frac{A_1}{\sqrt q}\right)^{\frac{1}{2s+1}}$, the
first and third terms of the right-hand side of \eqref{e.D} are both equal
to $s\left(\frac{\sqrt q}{A_1}\right)^{\frac{1}{2s+1}}$. For the remaining
term, note that $K\ge 2$ by our assumption on $A_1$ and thus $1+\log K\ll
\log K\ll s^{-1}\log q$. Here in the last estimate we used the fact $A_1<q$,
which implies $K<q^{\frac{1}{4s+2}}$. The bound \eqref{e.0} follows.

For \eqref{e.1}, we may assume $A_1A_2\ge q$ and, by symmetry, $A_1\ge \sqrt
q$. By \eqref{e.M2}, we get
\[D\ll \frac{1}{K}+K\sqrt{\frac{q}{A_1A_2}}.\]
We conclude by taking $K=\left(\frac{q}{A_1A_2}\right)^{-1/4}$.
\end{proof}

\begin{proof}[Proof of Corollary \ref{c}]
Part (ii) is obvious from \eqref{e.1}. For (i), it suffices to show the
existence of a constant $C'>0$ (depending on $\epsilon$ and $c$) such that
\[D(\cA_1\times \cA_2\times \cB)\le C'\left((\log\log
q)^{-1}+\sqrt{\frac{(\log
q)^{\frac{1}{\epsilon}-1}(\log\log q)}{A_2}}\right).
\]
To see this, we take $s=\lceil \frac{\epsilon\log q}{2\log\log q}\rceil$ in
\eqref{e.0} and note that there exist constants $c'$ and $c''$ (depending on
$\epsilon$ and $c$) such that
\begin{align*}
\log\left(s\left(\frac{\sqrt q}{A_1}\right)^{\frac{1}{2s+1}}\right)&\le \log s+\frac{\log(c^{-1}q^{-\epsilon})}{2s+1}\\
&\le (\log \log q-\log\log\log q)-\log\log q +c'\\
&=-\log\log\log q +c',\\
\log\left(\frac{1}{\sqrt{s}}\left(\frac{q}{A_1}\right)^{\frac{1}{2s}}\log q\right)&\le -\frac{1}{2}\log s+\frac{\log(c^{-1}q^{\frac{1}{2}-\epsilon})}{2s}+\log\log q\\
&\le -\frac{1}{2}(\log\log q-\log\log\log q)+\left(\frac{1}{2\epsilon}-1\right)\log\log q+\log\log q+c''\\
&=\left(\frac{1}{2\epsilon}-\frac{1}{2}\right)\log\log q+\frac{1}{2}\log\log\log q+c''.
\end{align*}
\end{proof}

\subsection*{Acknowledgments}
We wish to thank Ping Xi for many helpful comments, especially his
suggestion on the presentation of the proof of \eqref{e.M2}. We are grateful
to Nick Katz and Zhiyong Zheng for encouragement. We thank the referee for
useful comments. We acknowledge the support of Princeton University, where
part of this work was done during a visit.


\begin{bibdiv}
\begin{biblist}
\bib{dSSV}{article}{
   author={de la Bret\`eche, R\'{e}gis},
   author={Sha, Min},
   author={Shparlinski, Igor E.},
   author={Voloch, Jos\'{e} Felipe},
   title={The Sato-Tate distribution in thin parametric families of elliptic
   curves},
   journal={Math. Z.},
   volume={290},
   date={2018},
   number={3-4},
   pages={831--855},
   issn={0025-5874},
   review={\MR{3856834}},
   doi={10.1007/s00209-018-2042-0},
}

\bib{Deligne}{article}{
   author={Deligne, P.},
   title={Application de la formule des traces aux sommes
   trigonom\'etriques},
   book={
       title={Cohomologie \'etale},
       series={Lecture Notes in Mathematics},
       volume={569},
       note={S\'eminaire de G\'eom\'etrie Alg\'ebrique du Bois-Marie SGA
   4$\frac{1}{2}$},
       publisher={Springer-Verlag},
       place={Berlin},
       review={\MR{0463174 (57 \#3132)}},
       date={1977},
   },
   pages={168--232},
}

\bib{ET}{article}{
   author={Erd{\"o}s, P.},
   author={Tur{\'a}n, P.},
   title={On a problem in the theory of uniform distribution. I, II},
   journal={Nederl. Akad. Wetensch., Proc.},
   volume={51},
   date={1948},
   pages={1146--1154, 1262--1269 = Indagationes Math. \textbf{10}, 370--378, 406--413},
   review={\MR{0027895 (10,372c)}, \MR{0027896 (10,372d)}},
}

\bib{Katz}{book}{
   author={Katz, N. M.},
   title={Gauss sums, Kloosterman sums, and monodromy groups},
   series={Annals of Mathematics Studies},
   volume={116},
   publisher={Princeton University Press},
   place={Princeton, NJ},
   date={1988},
   pages={x+246},
   isbn={0-691-08432-7},
   isbn={0-691-08433-5},
   review={\MR{955052 (91a:11028)}},
}

\bib{KZ}{article}{
   author={Katz, N. M.},
   author={Zheng, Z.},
   title={On the uniform distribution of Gauss sums and Jacobi sums},
   conference={
      title={Analytic number theory, Vol.\ 2},
      address={Allerton Park, IL},
      date={1995},
   },
   book={
      series={Progr. Math.},
      volume={139},
      publisher={Birkh\"auser Boston},
      place={Boston, MA},
      date={1996},
   },
   pages={537--558},
   review={\MR{1409377 (97e:11089)}},
}

\bib{LZZ}{article}{
   author={Lu, Qing},
   author={Zheng, Weizhe},
   author={Zheng, Zhiyong},
   title={On the distribution of Jacobi sums},
   journal={J. Reine Angew. Math.},
   volume={741},
   date={2018},
   pages={67--86},
   issn={0075-4102},
   review={\MR{3836143}},
   doi={10.1515/crelle-2015-0087},
}



\bib{Shp}{article}{
   author={Shparlinski, I. E.},
   title={On the distribution of arguments of Gauss sums},
   journal={Kodai Math. J.},
   volume={32},
   date={2009},
   number={1},
   pages={172--177},
   issn={0386-5991},
   review={\MR{2518562 (2010b:11104)}},
   doi={10.2996/kmj/1238594554},
}

\bib{Xi}{article}{
   author={Xi, Ping},
   title={Equidistributions of Jacobi sums},
   note={arXiv:1809.04286v1, preprint},
}
\end{biblist}
\end{bibdiv}
\end{document}